\documentclass[a4paper]{amsart}
\usepackage{amssymb}
\usepackage{stmaryrd}
\usepackage[all]{xy}
\usepackage{epsf}
\usepackage{enumerate}
\newcommand{\printname}[1] {}

\newtheorem{theorem}{Theorem}[section]

\newtheorem{proposition}[theorem]{Proposition}
\newtheorem{lemma}[theorem]{Lemma}
\newtheorem{corollary}[theorem]{Corollary}

\newtheorem{definitions}[theorem]{Definitions}
\newtheorem{example}[theorem]{Example}
\newtheorem{remark}[theorem]{Remark}

\newcommand{\rmap}{\longrightarrow}

\usepackage{amsmath,amscd}

\begin{document}
\title{Formal equivalence of Poisson structures around Poisson submanifolds}
\author{Ioan M\v arcu\c t}
\address{Depart. of Math., Utrecht University, 3508 TA Utrecht,
The Netherlands}
\email{I.T.Marcut@uu.nl}
\begin{abstract}


Let $(M,\pi)$ be a Poisson manifold. A Poisson submanifold $P\subset M$ gives rise to an algebroid $A_P\to P$, to which we associate certain chomology groups
which control formal deformations of $\pi$ around $P$. Assuming that these groups vanish,  we prove that $\pi$ is formally rigid around $P$, i.e.\ any other Poisson structure on $M$, with the same first order jet along $P$ as $\pi$ is formally Poisson diffeomorphic to $\pi$.
When $P$ is a symplectic leaf, we find a list of criteria which imply that these cohomological obstructions vanish. In particular we obtain a formal version of the normal form theorem for Poisson manifolds around symplectic leaves.\\



\end{abstract}
\maketitle

\setcounter{tocdepth}{1}

\section{Introduction}

A \textbf{Poisson bracket} on a manifold $M$ is a Poisson algebra
structure on the space of smooth functions on $M$, i.e.\ Lie bracket
$\{\cdot,\cdot\}$ on $C^{\infty}(M)$ satisfying the derivation
property
\begin{equation}\label{derivation}
\{f,gh\}=\{f,g\}h+\{f,h\}g, \quad (\forall) f,g,h \in C^{\infty}(M).
\end{equation}
Equivalently, it can be given by a bivector $\pi\in\mathfrak{X}^2(M)$ which satisfies $[\pi,\pi]=0$. The two definitions are related by the
formula:
\[\langle\pi,df\wedge dg\rangle=\{f,g\},\quad (\forall) f,g \in C^{\infty}(M).\]
An immersed submanifold $\iota:P\to M$ is called a \textbf{Poisson submanifold} of $M$ if $\pi$ is tangent to $P$. This insures that $\pi_{|P}$
is a Poisson structure on $P$ for which restriction map
\[\iota^*:C^{\infty}(M) \to C^{\infty}(P)\]
is a Lie algebra homomorphism. We regard the Poisson algebra $(C^{\infty}(P),\{\cdot,\cdot\})$ as the $0$-th order approximation of the Poisson
structure on $M$. If $P$ is embedded, we have that $P$ is a Poisson submanifold if and only if its vanishing ideal \[I(P)=\{f\in C^{\infty}(M) |
\iota^*(f)=0\}\] is an ideal in the Lie algebra $(C^{\infty}(M),\{\cdot,\cdot\})$. Assuming that $P$ is also closed \footnote{Since we study
local properties of $(M,\pi)$ around $P$, only the condition that $P$ is embedded is essential, closeness can be achieved by replacing $M$ with
a tubular neighborhood of $P$.}, we have a canonical identification of Poisson algebras
\[(C^{\infty}(P),\{\cdot,\cdot\})=(C^{\infty}(M)/I(P),\{\cdot,\cdot\}).\]
This gives a recipe for construction higher order approximations, for example the first order approximation fits into an exact sequence of
Poisson algebras
\begin{equation}\label{shortexact}
0\to (I(P)/I^{2}(P),\{\cdot,\cdot\})\to(C^{\infty}(M)/I^{2}(P),\{\cdot,\cdot\})\to (C^{\infty}(P),\{\cdot,\cdot\}) \to 0.
\end{equation}
The Poisson algebra structures in this sequence depend only on $j^1_{|P}\pi$, the first jet of $\pi$ along $P$. A better way to describe
(\ref{shortexact}) is using Lie algebroids, as will be explained in section \ref{The first order data}, the sequence (\ref{shortexact}) gives
rise to a Lie algebroid structure on $T^*_PM$, which is will be denoted by $A_P$ and is an extension of the form:
\begin{equation}\label{shortexact2}
0\to TP^{\circ}\to A_P\to T^*P\to 0,
\end{equation}
where $TP^{\circ}\subset T^*_PM=A_P$ is the annihilator of $TP$ and $T^*P$ is the cotangent algebroid of $(P,\{\cdot,\cdot\})$. In particular we
obtain a representation of $A_P$ on $TP^{\circ}$ and thus also on its symmetric powers $\mathcal{S}^{k}(TP^{\circ})$.

In this paper we study formal equivalence of Poisson structure around Poisson submanifolds, more precisely we obtain the following result:
\begin{theorem}\label{Proposition_formal_equivalence}
Let $\pi_1$ and $\pi_2$ be two Poisson structures on $M$, such that $P\subset M$ is an embedded Poisson submanifold for both, and such that they
have the same first jet along $P$. If their common algebroid $A_P$ has the property that
\[H^{2}(A_P;\mathcal{S}^{k}(TP^{\circ}))=0, \ \ (\forall) \  k\geq 2,\]
then the two structures are formally Poisson diffeomorphic. More precisely, there exists a diffeomorphism
\[\psi:\mathcal{U}\to \mathcal{V}, \ \ \psi_{|P}=id_P \ \,d\psi_{|T_PM}=id_{T_PM},\]
where $\mathcal{U}$ and $\mathcal{V}$ are open neighborhoods of $P$, such that $\pi_{1|\mathcal{U}}$ and $\psi^*(\pi_{2|\mathcal{V}})$ have the
same infinite jet along $P$:
\[j_{|P}^{\infty}(\pi_{1|\mathcal{U}})=j^{\infty}_{|P}(\psi^*(\pi_{2|\mathcal{V}})).\]
\end{theorem}

Applying Theorem \ref{Proposition_formal_equivalence} to the linear Poisson structure on the dual of a compact, semi-simple Lie algebra, we
obtain the following result.

\begin{corollary}\label{C_Dual_ss_Lie}
Let $\mathfrak{g}$ be a semi-simple Lie algebra of compact type and consider $\pi_{\mathrm{lin}}$ the linear Poisson structure on
$\mathfrak{g}^*$. Let $\mathbb{S}({\mathfrak{g}})\subset \mathfrak{g}^*$ be the sphere in $\mathfrak{g}^*$  centered at 0, of radius 1 with
respect to some invariant inner product. Then $\mathbb{S}({\mathfrak{g}})$ is a Poisson submanifold and any Poisson structure $\pi_1$ defined in
some open neighborhood of $\mathbb{S}({\mathfrak{g}})$, such that
\[j^1_{|P}(\pi_{\mathrm{lin}})=j^1_{|P}(\pi_1),\]
is formally Poisson diffeomorphic to $\pi_{\mathrm{lin}}$.
\end{corollary}

A special type of Poisson submanifolds are the \textbf{symplectic leaves} of $(M,\pi)$. A symplectic leaf of $M$ is a connected Poisson
submanifold $S$ of $M$, for which $\pi_{|S}$ is nondegenerate, and moreover $S$ is maximal with these properties. Any Poisson manifold carries a
natural singular foliation by symplectic leaves and $\omega_S:=\pi_{|S}^{-1}$ gives a symplectic structure on such an $S$. If
$(S,\omega_S)\subset (M,\pi)$ is an embedded symplectic leaf, then the Lie algebroid extension (\ref{shortexact2}) - which encodes only the
first order jet $\pi$ along $S$ - can be used to construct a second Poisson structure $\pi^1_S$, called the \textbf{first order approximation}
of $\pi$ around $S$. $\pi^1_S$ is defined on some open neighborhood of $S$ and has the same first jet as $\pi$ along $S$.

M.Crainic and the author have obtained in \cite{CrMa} a normal form theorem for Poisson structure around symplectic leaves, proving that under
appropriate conditions on the first jet of $\pi$ along $S$, $\pi$ and $\pi^1_S$ are Poisson diffeomorphic. The goal of this research is to give
a formal version of this result, which in its most general form we state below (observe that it is a direct consequence of Theorem
\ref{Proposition_formal_equivalence}).

\begin{theorem}\label{Theorem1}
Let $(M,\pi)$ be a Poisson manifold and $S\subset M$ an embedded symplectic leaf. If the cohomology groups
\[H^2(A_S,\mathcal{S}^{k}(TS^{\circ}))\]
vanish for all $k\geq 2$, then $\pi$ is formally Poisson diffeomorphic to its first order approximation around $S$.
\end{theorem}

In many cases we show that these cohomological obstructions vanish, and we obtain the following list of corollaries.

\begin{corollary}\label{Theorem2}
Let $(M,\pi)$ be a Poisson manifold and $S\subset M$ an embedded symplectic leaf. Assume that the Poisson homotopy cover of $S$ is a smooth
principal bundle with vanishing second DeRham cohomology group, and 
its structure group $G$ satisfies
\[H^2_{\mathrm{diff}}(G,\mathcal{S}^k(\mathfrak{g}))=0,\ \ (\forall)\  k\geq 2,\]
where $\mathfrak{g}$ is the Lie algebra of $G$, and $H^{*}_{\mathrm{diff}}(G,\mathcal{S}^k(\mathfrak{g}))$ denotes the differentiable cohomology
of $G$ with coefficients in the $k$-th symmetric power of the adjoint representation. Then $\pi$ is formally Poisson diffeomorphic to its first
order approximation around $S$.
\end{corollary}

Since the differentiable cohomology of compact groups vanishes, we obtain the following immediate corollary.
\begin{corollary}\label{Corollary_2}
Let $(M,\pi)$ be a Poisson manifold and $S\subset M$ an embedded symplectic leaf. If the Poisson homotopy cover of $S$ is a smooth principal
bundle with vanishing second DeRham cohomology group and compact structure group, then $\pi$ is formally Poisson diffeomorphic to its first
order approximation around $S$.
\end{corollary}

A bit more technical is the following.
\begin{corollary}\label{Theorem3}
Let $(M, \pi)$ be a Poisson manifold and $S\subset M$ be an embedded symplectic leaf whose isotropy Lie algebra is reductive. If the
abelianization algebroid
\[A_S^{\mathrm{ab}}:=A_S/[TS^{\circ},TS^{\circ}]\]
is integrable by a simply connected principal bundle with vanishing second DeRham cohomology group and compact structure group, then $\pi$ is
formally Poisson diffeomorphic to its first order approximation around $S$.
\end{corollary}

\begin{corollary}\label{Corollary_3}
Let $(M, \pi)$ be a Poisson manifold and $S\subset M$ be an embedded symplectic leaf through $x\in M$. If the isotropy Lie algebra at $x$ is
semi-simple, $\pi_1(S,x)$ is finite and $\pi_2(S,x)$ is torsion, then $\pi$ is formally Poisson diffeomorphic to its first order approximation
around $S$.
\end{corollary}

\subsection*{Some related results}
The first order approximation of a Poisson manifold $(M,\pi)$ around a one-point leaf $x$ (a zero of $\pi$), is the linear Lie-Poisson structure
on $\mathfrak{g}_x^*$, the dual of the isotropy Lie algebra at $x$. Formal linearization in this setup was proven by Weinstein in \cite{Wein}
for semi-simple $\mathfrak{g}_x$. This case is covered also by our Corollary \ref{Corollary_3}. Under the stronger assumption that
$\mathfrak{g}_x$ is semi-simple of compact type, J.Conn proves in \cite{Conn} that a neighborhood of $x$ is in fact Poisson diffeomorphic to an
open neighborhood of $0$ in the local model $\mathfrak{g}_x^*$.

The first order approximation around arbitrary symplectic leaves was constructed by Vorobjev in \cite{Vorobjev} (we recommend also \cite{CrMa}
for a more geometrical approach).

A weaker version of our Theorem \ref{Proposition_formal_equivalence} - of which we become aware only at the end of this research - was stated in
\cite{VorCo}. Instead of embedded Poisson submanifolds, the authors of \cite{VorCo} work around compact symplectic leafs and also their
conclusion is a bit weaker, they prove for each $k$ existence of a diffeomorphism which identifies the Poisson structures up to order $k$ (see
Theorem 7.1 in \cite{VorCo}).
Compactness of the leaf is nevertheless a too strong assumption of formal equivalence. For example in Corollary 7.4 in \cite{VorCo}, the authors
conclude that hypotheses similar to those in our Corollary \ref{Corollary_3} imply the vanishing to the cohomology groups
$H^2(A_S,\mathcal{S}^{k}(TS^{\circ}))$, but on the other hand they also remark that compactness of the leaf is incompatible with these
assumptions (they force $S$ to be a point).

To prove Theorem \ref{Proposition_formal_equivalence} we first reduce it to a result on equivalence of Maurer Cartan elements in complete graded
Lie algebras, whose proof we give in the Appendix. The same criteria for equivalence of Maurer Cartan elements, but in the context of
differential graded algebras can be found in the Appendix A of \cite{CMB}.

To prove vanishing of the cohomological obstructions - and obtain the list of corollaries enumerated above - we use techniques like Whitehead's
Lemma for semi-simple Lie algebras, spectral sequences for Lie algebroids, but also the more powerful techniques developed in \cite{Cra} such as
Van Est isomorphism and vanishing of cohomology of proper groupoids.

The main result of \cite{CrMa} is the following normal form theorem:\\

\noindent\emph{Let $(M,\pi)$ be a Poisson manifold and $S\subset M$ an embedded symplectic leaf. If the following conditions are satisfied
\begin{itemize}
\item the Poisson homotopy cover $P$ of $S$ is smooth,
\item $H^2_{\mathrm{dR}}(P)=0$,
\item the structure group of $P$ is compact,
\item $S$ is compact,
\end{itemize}
then $\pi$ is Poisson diffeomorphic to its first order approximation around $S$.}\\

The first 3 conditions are the hypothesis of our Corollary \ref{Corollary_2}. So, giving up on compactness of the leaf, we still can conclude
that $\pi$ and its first order approximation are \textbf{formally} Poisson diffeomorphic. Nevertheless, the conditions of Corollary
\ref{Corollary_2} are too strong in the formal setting, they force the semi-simple part of the isotropy Lie algebra to be compact. This is why
the author considers the more technical Corollary \ref{Theorem3} to be the correct analog in the category of formal power series of the normal
form theorem in \cite{CrMa}. In fact Corollary \ref{Corollary_2} follows also from Corollary \ref{Theorem3}, it is precisely the case when the
semi-simple part of the isotropy Lie algebra is compact.\\

%

\noindent \textbf{Acknowledgments.} I would like to thank Marius Crainic for his very useful suggestions and comments.\\

\noindent This research was supported by the NWO Vidi Project ``Poisson topology''.

\section{The first order data}\label{The first order data}

We start by recalling some definitions, for more on Lie algebroids see \cite{Mackenzie}.
\begin{definitions}
A \textbf{Lie algebroid} over a manifold $B$ is a vector bundle $\mathcal{A}\to B$ endowed with a Lie bracket $[\cdot, \cdot]$ on its space of
sections $\Gamma(\mathcal{A})$ and a vector bundle map $\rho: \mathcal{A}\to TB$, called the \textbf{anchor}, which satisfy the Leibniz
identity:
\[ [\alpha, f\beta]= f[\alpha, \beta]+ \mathcal{L}_{\rho(\alpha)}(f)\beta\ \ \ (\forall) \ f\in C^{\infty}(B), \ \alpha, \beta \in \Gamma(\mathcal{A}).\]

A \textbf{representation} of $\mathcal{A}$ is a vector bundle $E\to B$ endowed  with a bilinear map
\[\nabla:\Gamma(\mathcal{A})\times\Gamma(E)\to \Gamma(E),\]
satisfying
\[\nabla_{f\alpha}(s)=f\nabla_{\alpha}(s),\ \ \nabla_{\alpha}(fs)=f\nabla_s(\alpha)+\mathcal{L}_{\rho(\alpha)}(f)s,\]
and the flatness condition
\[\nabla_{\alpha}\nabla_{\beta}(s)-\nabla_{\beta}\nabla_{\alpha}(s)=\nabla_{[{\alpha},{\beta}]}(s).\]

 The \textbf{cohomology of an algebroid}
$(\mathcal{A},[\cdot,\cdot],\rho)$ \textbf{with coefficients in a representation} $(E,\nabla)$ is defined by the complex
$\Omega^{\bullet}(\mathcal{A},E):=\Gamma(\Lambda^{\bullet}\mathcal{A}^*\otimes E)$ with differential given by the classical Koszul formula:
\begin{align*}
d_{\nabla}\omega(\alpha_0, \ldots , \alpha_{q})=&\sum_{i}(-1)^{i} \nabla_{\alpha_i}(\omega(\alpha_1, \ldots , \widehat{\alpha}_i, \ldots , \alpha_{q}))+\\
& + \sum_{i< j} (-1)^{i+j}\omega([\alpha_i, \alpha_j], \ldots , \widehat{\alpha}_i, \ldots, \widehat{\alpha}_j, \ldots , \alpha_{q}).
\end{align*}
The corresponding cohomology groups will be denoted by $H^{\bullet}(\mathcal{A},E)$.
\end{definitions}

For any Poisson manifold $(M,\pi)$ there is an associated algebroid structure on the cotangent bundle $T^*M$, with anchor given by
$\pi^{\sharp}:T^*M\to TM$ and bracket uniquely determined by (for details see \cite{Izu})
\[[df,dg]:=d\{f,g\},\ \  (\forall)f,g\in C^{\infty}(M).\]
Let $P\subset M$ be an embedded Poisson submanifold. Since $\pi$ is tangent to $P$ it is easy to see that the algebroid structure can be
restricted to $P$, in the sense that there is a unique algebroid structure on $A_P:=T^*_PM$ with anchor $\pi^{\sharp}_{|P}$
and bracket such that the restriction map $\Gamma(T^*M)\to \Gamma(A_P)$ is a Lie algebra homomorphism. The dual of the inclusion $TP\subset
T_PM$ gives a map $A_P\to T^*P$ which is a Lie algebroid homomorphism, where $T^*P$ is the cotangent algebroid of $(P,\pi_{|P})$. This way we
obtain the extension of Lie algebroids from the introduction:
\begin{equation}\label{shortexact_alg}
0\to (TP^{\circ},[\cdot,\cdot])\to (A_P,[\cdot,\cdot]) \to (T^*P,[\cdot,\cdot])\to 0.
\end{equation}
This short exact sequence shows also that $TP^{\circ}$ is an ideal in $(A_P,[\cdot,\cdot])$, therefore
\[\nabla:\Gamma(A_P)\times \Gamma(TP^{\circ})\to \Gamma(TP^{\circ}), \ \ \nabla_{\alpha}(\eta):=[\alpha,\eta],\]
defines a representation of $A_P$ on $TP^{\circ}$, and thus on its symmetric powers $\mathcal{S}^{k}(TP^{\circ})$. The resulting cohomology
groups are the obstructions appearing in Theorems \ref{Proposition_formal_equivalence} and \ref{Theorem1}. The algebroid structures on $A_P$ and
the sequence (\ref{shortexact_alg}) depend only on the first jet of $\pi$ along $P$ (i.e. the brackets and anchors can be expressed in terms of
$\pi_{|P}$ and the first order derivatives of $\pi$ restricted to $P$).
\begin{remark}\label{Proposition_first_order_data}
We regard the algebroid $A_P$ as the first order approximation of the Poisson bracket at $P$. To justify this interpretation, fix a Poisson
structure $\pi_P$ on $P$, where $P\subset M$ a closed embedded submanifold. Then there is a one to one correspondence between Poisson algebra
structures on the commutative algebra $C^{\infty}(M)/I^2(P)$, which fit into the short exact sequence
\begin{equation}\label{sh_ex_Poi}
0\to (I(P)/I^{2}(P),\{\cdot,\cdot\})\to(C^{\infty}(M)/I^{2}(P),\{\cdot,\cdot\})\to (C^{\infty}(P),\{\cdot,\cdot\}) \to 0,
\end{equation}
and algebroid structures on $A_P:=T^*_PM$ which fit into a sequence of the form
\begin{equation}\label{sh_ex_alg}
0\to (TP^{\circ},[\cdot,\cdot])\to (A_P,[\cdot,\cdot]) \to (T^*P,[\cdot,\cdot])\to 0.
\end{equation}
The exterior derivative induces a map
\[d: C^{\infty}(M)/I^2(P)\to \Gamma(A_P),\]
and the correspondence between the brackets is uniquely determine by the fact that this is a Lie algebra homomorphism.
\end{remark}

\begin{example}
Consider $P:=\mathbb{R}^2$ as the submanifold $\{z=0\}\subset M:=\mathbb{R}^3$. We construct a first order extension of the trivial Poisson
structure on $P$ to $M$, i.e. a Poisson algebra structure on the commutative algebra
\[C^{\infty}(M)/I(P)^2=C^{\infty}(M)/(z^2)=C^{\infty}(P)\oplus z C^{\infty}(P)\]
with the property that $\{f,g\}\in (z)$, for all $f,g\in C^{\infty}(M)/(z^2)$. Explicitly, define
\begin{align*}
\{f,g\}&=z(\frac{\partial{f}}{\partial x}\frac{\partial{g}}{\partial y}-\frac{\partial{f}}{\partial y}\frac{\partial{g}}{\partial
x}+x\frac{\partial{f}}{\partial x}\frac{\partial{g}}{\partial z}-x\frac{\partial{f}}{\partial z}\frac{\partial{g}}{\partial x}) \ modulo \
(z^2).
\end{align*}
A straightforward computation yields that $\{\cdot,\cdot\}$ satisfies the Jacobi identity 
and thus we have an extension of Poisson algebras
\[0\to z C^{\infty}(P) \to C^{\infty}(P)\oplus z C^{\infty}(P) \to C^{\infty}(P)\to 0,\]
where the Poisson bracket on $P$ is zero. The corresponding algebroid is on $A_P=\mathbb{R}^3\times P\to P$, it has zero anchor and bracket
defined by
\[[dx_{|P},dy_{|P}]=dz_{|P}, \ \ [dy_{|P},dz_{|P}]=0,\ \ [dx_{|P},dz_{|P}]=xdz_{|P},\]
and extended bi-linearly to all sections, since the anchor is trivial.

Nevertheless, there is no Poisson structure on $M$ (nor on any open neighborhood of $P$), for which this bracket is the first order
approximation. Assume on the contrary that on some open neighborhood $\mathcal{U}$ of $P$ such a Poisson structure exists. Then is must have the
following form:
\[\{x,y\}=z+z^2h,\ \ \{y,z\}=z^2k, \ \ \{x,z\}=xz+z^2l,\]
for some smooth functions $h,k,l$ defined on $\mathcal{U}$. Computing the Jacobiator of $x$, $y$, $z$, we obtain
\[J=\{x,\{y,z\}\}+\{z,\{x,y\}\}+\{y,\{z,x\}\}=z^2((2-x)k(x,y,0)+1)+z^3a(x,y,z),\]
where $a$ is a smooth function. In particular we see that $J$ cannot vanish since:
\[\frac{\partial^2J}{\partial z^2}(2,y,0)=2\neq 0.\]
\end{example}
The example shows that not everything that looks like the first jet of a Poisson structure around $P$ (i.e. an extension of the form
(\ref{sh_ex_alg}) or (\ref{sh_ex_Poi})) comes from an actual Poisson structure.

On the other hand if $P$ is a symplectic manifold, the situation changes for better, namely every "first jet" of a Poisson structure can be
extended to a Poisson structure around $P$. More precisely, consider $(S,\omega_S)$ a symplectic manifold, with $S\subset M$ embedded, and an
algebroid structure on $A_S:=T^*_SM$ which fits into an exact sequence of the form
\[0\to  TS^{\circ} \to A_S \to T^*S \to 0.\]
Then, using a tubular neighborhood $\mathcal{E}:T_SM/TS\to M$, one can construct a Poisson structure $\pi_S^1=\pi_S^1(A_S,\omega_S,\mathcal{E})$
on some open neighborhood of $S$, from which we recover the first order data, i.e.\ it has $(S,\omega_S)$ as a symplectic leaf and the algebroid
structure induced on $T^*_SM$ is $A_S$. This Poisson structure was first constructed by Vorobjev, the reader can find this construction in
\cite{Vorobjev} and we also recommend \cite{CrMa} for some different approaches. The construction applied to different tubular neighborhoods
produces Poisson structures which, when restricted to small enough neighborhoods of $S$, are Poisson diffeomorphic (see \cite{Vorobjev}). So the
isomorphism class of the germ around $S$ of $\pi_S^1$ doesn't depend on $\mathcal{E}$.

We can view the whole story from a different perspective, namely start with a Poisson structure $\pi$ on $M$, for which $(S,\omega_S)$ is an
embedded symplectic leaf, and denote as usually by $A_S$ the algebroid on $T^*_SM$. For $\mathcal{E}$ a tubular neighborhood of $S$, we will
call $\pi_S^1=\pi_S^1(A_S,\omega_S,\mathcal{E})$ \textbf{the first order approximation} of $\pi$ around $S$, and is defined on some open $U$
with $S\subset U$. The first order approximation plays the role of a normal form for $\pi$ around $S$.

\section{The formal equivalence Theorem}\label{Formal equivalence of Poisson structures around symplectic
leaves}

\subsection*{The algebra of formal vector fields}


Consider the graded Lie algebra of multi-vector fields on $M$, $(\mathfrak{X}^{\bullet}(M),[\cdot,\cdot])$ with Lie bracket the
Nijenhuis-Schouten bracket and $deg(W)=k-1$ for $W\in \mathfrak{X}^{k}(M)$. For a closed, embedded submanifold $P\subset M$, denote by
$\mathfrak{X}^{\bullet}_P(M)$ the subalgebra of multi-vector fields tangent to $P$
\[\mathfrak{X}^{\bullet}_P(M):=\{u\in \mathfrak{X}^{\bullet}(M)|u_{|P}\in \mathfrak{X}^{\bullet}(P)\}.\]
The vanishing ideal of $P$, $I(P)\subset C^{\infty}(M)$ induces a filtration $\mathcal{F}$ on $\mathfrak{X}^{\bullet}_P(M)$:
\[\mathfrak{X}^{\bullet}_P(M)\supset \mathcal{F}^{\bullet}_0\supset\mathcal{F}^{\bullet}_1\supset\ldots \mathcal{F}^{\bullet}_k\supset\mathcal{F}^{\bullet}_{k+1}\supset\ldots.\]
\[\mathcal{F}^{\bullet}_k=I^{k+1}(P)\mathfrak{X}^{\bullet}(M),\ \   k\geq 0.\]
It is readily checked that
\begin{equation}\label{Filtration_bracket}
[\mathcal{F}_k,\mathcal{F}_l]\subset \mathcal{F}_{k+l},\ \ [\mathfrak{X}_P^{\bullet}(M),\mathcal{F}_k]\subset \mathcal{F}_k.
\end{equation}
Let $\hat{\mathfrak{X}}^{\bullet}_P(M)$ be the completion of $\mathfrak{X}^{\bullet}_P(M)$ with respect to the filtration $\mathcal{F}$, defined
by the projective limit
\[\hat{\mathfrak{X}}^{\bullet}_P(M):=\varprojlim \mathfrak{X}^{\bullet}_P(M)/\mathcal{F}^{\bullet}_{k}.\]
By (\ref{Filtration_bracket}) it follows that $\hat{\mathfrak{X}}^{\bullet}_P(M)$ inherits a graded Lie algebra structure, such that the natural
maps
\[ j^{k}_{|P}:\hat{\mathfrak{X}}^{\bullet}_P(M)\to \mathfrak{X}^{\bullet}_P(M)/\mathcal{F}^{\bullet}_k,\textrm{ for }k\geq 0\]
are Lie algebra homomorphisms. The algebra $(\hat{\mathfrak{X}}^{\bullet}_P(M),[\cdot,\cdot])$ will be called \textbf{the algebra of formal
multi-vector fields along $P$}. Consider also the homomorphism
\[j^{\infty}_{|P}:\mathfrak{X}^{\bullet}_P(M)\to \hat{\mathfrak{X}}^{\bullet}_P(M).\]
From a version of Borel's Theorem (see for example \cite{Moerdijk}) about existence of smooth section with a specified infinite jet along a
submanifold, it follows that $j^{\infty}_{|P}$ is surjective. Observe that $\hat{\mathfrak{X}}^{\bullet}_P(M)$ inherits a filtration
$\hat{\mathcal{F}}$ from $\mathfrak{X}^{\bullet}_P(M)$, given by
\[\hat{\mathcal{F}}^{\bullet}_{k}=j^{\infty}_{|P}\mathcal{F}^{\bullet}_{k},\]
and which satisfies the corresponding equations (\ref{Filtration_bracket}).

The adjoint action of an element $X\in\hat{\mathcal{F}}_1^1$
\[ad_X: \hat{\mathfrak{X}}^{\bullet}_P(M)\to \hat{\mathfrak{X}}^{\bullet}_P(M),\ \ ad_X(Y):=[X,Y]\]
increases the degree of the filtration by 1. Therefore the partial sums
\[\sum_{i=0}^n\frac{ad_{X}^i}{i!}(Y)\]
are constant modulo $\hat{\mathcal{F}}_k$ for $n\geq k$ and all $Y\in\hat{\mathfrak{X}}^{\bullet}_P(M)$. This and the completeness of the
filtration on $\hat{\mathcal{F}}$ show that the exponential of $ad_X$
\[e^{ad_X}:\hat{\mathfrak{X}}^{\bullet}_P(M)\to \hat{\mathfrak{X}}^{\bullet}_P(M),\ \ e^{ad_{X}}(Y):=\sum_{n\geq 0} \frac{ad_{X}^n}{n!}(Y)\]
is well defined. It is readily checked that $e^{ad_X}$ is a graded Lie algebra isomorphism with inverse $e^{-ad_X}$ and that it preserves the
filtration. We will need the following geometric interpretation of these isomorphisms.
\begin{lemma}\label{gauge_real}
For every $X\in\hat{\mathcal{F}}_1^1$, there exists $\psi:M\to M$ a diffeomorphism of $M$, with $\psi_{|P}=id_P$ and $d\psi_{|P}=id_{T_PM}$,
such that for every $W\in \mathfrak{X}^{\bullet}_P(M)$, we have that
\[j^{\infty}_{|P}(\psi^*(W))=e^{ad_X}(j^{\infty}_{|P}(W)).\]
\end{lemma}
\begin{proof}
By Borel's Theorem there exists a vector field $V$ on $M$, such that $X=j^{\infty}_{|P}(V)$. We claim that $V$ can be chosen to be complete. Let
$g$ be a complete metric on $M$ and let $\phi:M\to [0,1]$ be a smooth function, such that $\phi=1$ on the set $\{x|g_x(V_x,V_x)\leq 1/2\}$ and
$\phi=0$ on the set $\{x|g_x(V_x,V_x)\geq 1\}$. Since $V_{|P}=0$, it follows that $\phi V$ has the same germ as $V$ around $P$, therefore
$j^{\infty}_{|P}(\phi V)=X$. On the other hand, since $\phi V$ is bounded it is complete, so replace $V$ by $\phi V$.

We will show that $\psi:=\mathrm{Fl}_V$, the flow of $V$ at time 1 satisfies all requirements. Since $j^1_{|P}(V)=0$, it is clear that
$\psi_{|P}=id_P$ and $d\psi_{|P}=id_{T_PM}$.

Let $W\in \mathfrak{X}_P^{\bullet}(M)$ and denote $W_s:=\mathrm{Fl}_{sV}^{*}(W)$ - the pullback of $W$ by the flow of $V$ at time $s$. Since
$W_s$ satisfies the differential equation $\frac{d}{ds}W_s=[V,W_s]$, a simple computation gives the following formula
\[\frac{d}{ds}(\sum_{i=0}^k\frac{(-s)^iad_{V}^i}{i!}(W_s))=\frac{(-s)^k ad_V^{k+1}}{k!}(W_s).\]
This shows that the sum
\[\sum_{i=0}^k\frac{(-s)^iad_{V}^i}{i!}(W_s)\]
modulo $\mathcal{F}_{k+1}$ is independent of $s$, therefore
\[W-\sum_{i=0}^k\frac{(-1)^iad_{V}^i}{i!}(\psi^*(W))\in\mathcal{F}_{k+1}.\]
Applying $j^{\infty}_{|P}$ to this equation yields
\[j^{\infty}_{|P}(W)-\sum_{i=0}^k\frac{(-1)^iad_{X}^i}{i!}j^{\infty}_{|P}(\psi^*(W))\in\hat{\mathcal{F}}_{k+1},\]
hence, the conclusion
\[j^{\infty}_{|P}(W)=e^{-ad_X}j^{\infty}_{|P}(\psi^*(W)).\]
\end{proof}

\subsection*{The cohomology of the restricted algebroid}
Let $(M,\pi)$ be a Poisson manifold and $P\subset M$ a closed embedded Poisson submanifold. The cohomologies we are considering are all versions
of the Poisson cohomology $H^{\bullet}_{\pi}(M)$, computed by the complex $\mathfrak{X}^{\bullet}(M)$ of multi-vector fields on $M$ and
differential $d_{\pi}=[\pi,\cdot]$, where $[\cdot,\cdot]$ is the Nijenhuis-Schouten bracket. Since $P$ is a Poisson submanifold we have that
$[\pi,I(P)\mathfrak{X}^{\bullet}(M)]\subset I(P)\mathfrak{X}^{\bullet}(M)$, and more general, it follows that
$I^{k}(P)\mathfrak{X}^{\bullet}(M)$ forms a subcomplex. Taking consecutive quotients, we obtain the following complexes
\[(I^{k}(P)\mathfrak{X}^{\bullet}(M)/I^{k+1}(P)\mathfrak{X}^{\bullet}(M),d^k_{\pi}),\]
with differential $d^k_{\pi}$ induced by $[\pi,\cdot]$. Observe that the differential on these complexes depends only on the first jet of $\pi$
along $P$, therefore, following the philosophy of section \ref{The first order data}, it can be described only in terms of the algebroid $A_P$.

\begin{proposition}\label{Proposition_isomorphic_complexes}
The following two complexes are isomorphic
\[(I^{k}(P)\mathfrak{X}^{\bullet}(M)/I^{k+1}(P)\mathfrak{X}^{\bullet}(M),d_\pi^k)\cong (\Omega^{\bullet}(A_P,\mathcal{S}^k(TP^{\circ})),d_{\nabla^k}),\ \ (\forall) k\geq 0.\]
\end{proposition}
\begin{proof}
Since the space of sections of $TP^{\circ}$ is spanned by differentials of elements in $I(P)$, it is easy to see that the map given by
\begin{align*}
\tau_k&:I^{k}(P)\mathfrak{X}^{\bullet}(M)\to \Omega^{\bullet}(A_P,\mathcal{S}^k(TP^{\circ}))=\Gamma(\Lambda^{\bullet}(T_PM)\otimes\mathcal{S}^{k}(TP^{\circ})),\\
&\tau_k(f_1\ldots f_kW)=W_{|P}\otimes df_{1|P} \odot \ldots \odot df_{k|P},
\end{align*}
where $f_1,\ldots,f_k\in I(P)$ and $W\in\mathfrak{X}^{\bullet}(M)$ is well defined and surjective. Moreover its kernel is precisely
$I^{k+1}(P)\mathfrak{X}^{\bullet}(M)$. Hence we are left to prove that
\begin{equation}\label{Commuting_with_differential}
\tau_k([\pi,W])=d_{\nabla^k}(\tau_k(W)),\ \ (\forall) W\in I^k(P)\mathfrak{X}^{\bullet}(M).
\end{equation}
Recall that the algebroid $A_P$ has anchor $\rho=\pi^{\sharp}_{|P}$ and bracket determined by
\[[d\phi_{|P},d\psi_{|P}]_P:=d\{\phi,\psi\}_{|P},\ \ (\forall) \phi,\psi\in C^{\infty}(M).\]
Moreover, for $k=0$, we have that $\nabla^0$ is given by
\[\nabla^0:\Gamma(A_P)\times C^{\infty}(P)\to C^{\infty}(P),\ \ \nabla^0_{\eta}(h)=\mathcal{L}_{\rho(\eta)}(h).\]
Since both differentials $d_{\pi}$ and $d_{\nabla^k}$ act by derivations and $\nabla^k$ is obtained by extending $\nabla^1$ by derivations, it
suffices to prove (\ref{Commuting_with_differential}) for $\phi\in C^{\infty}(M)$ and $X\in\mathfrak{X}^{1}(M)$ (with $k=0$), and for $f\in
I(P)$ (with $k=1$).\\
Let $\phi\in C^{\infty}(M)$ and $\eta\in\Gamma(A_P)$. Using that $\pi$ is tangent to $P$, we obtain
\[\tau_0([\pi,\phi])(\eta)=[\pi,\phi]_{|P}(\eta)=d\phi_{|P}(\pi^{\sharp}_{|P}(\eta))=\mathcal{L}_{\rho(\eta)}(\tau_0(\phi))=d_{\nabla^0}(\tau_0(\phi))(\eta).\]
Let $X\in \mathfrak{X}^{1}(M)$, $\phi,\psi\in C^{\infty}(M)$ and $\eta:=d\phi_{|P}, \theta:=d\psi_{|P}\in\Gamma(A_P)$. Then
\begin{align*}
\tau_0([\pi,X])&(\eta,\theta)=[\pi,X]_{|P}(d\phi_{|P},d\psi_{|P})=\\
&=(\{X(\phi),\psi\}+\{\phi,X(\psi)\}-X(\{\phi,\psi\}))_{|P}=\\
&=\pi^{\sharp}_{|P}(d\phi_{|P})(X_{|P}(d\psi_{|P}))-\pi^{\sharp}_{|P}(d\psi_{|P})(X_{|P}(d\phi_{|P}))-X_{|P}(d\{\phi,\psi\}_{|P})=\\
&=\mathcal{L}_{\rho(\eta)}(\tau_0(X)(\theta))-\mathcal{L}_{\rho(\theta)}(\tau_0(X)(\eta))-\tau_0(X)([\eta,\theta]_P)=\\
&=d_{\nabla^0}(\tau_0(X))(\eta,\theta),
\end{align*}
thus (\ref{Commuting_with_differential}) holds for $X$.\\
Consider now $f\in I(P)$ and $\eta:=d\phi_{|P}\in\Gamma(A_P)$, with $\phi\in C^{\infty}(M)$. The formula defining $\tau_k$ implies that for
every $W\in I^{k}(P)\mathfrak{X}^{\bullet}(M)$, we have that
\[\tau_k(i_{d\phi}(W))=i_{d\phi_{|P}}\tau_k(W).\]
Using this, the following finishes the proof
\begin{align*}
\tau_1([\pi,f])(\eta)&=\tau_1([\pi,f](d\phi))=\tau_1(\{\phi,f\})=d\{\phi,f\}_{|P}=\\
&=[\eta,df_{|P}]_P=\nabla^1_{\eta}(\tau(f))=d_{\nabla^1}(\tau(f))(\eta).
\end{align*}
\end{proof}

\subsection*{Proof of Theorem \ref{Proposition_formal_equivalence}}

By replacing $M$ with a tubular neighborhood of $P$, we can assume that $P$ is closed in $M$. Denote \[\gamma:=j^{\infty}_{|P}\pi_1,
\gamma':=j^{\infty}_{|P}\pi_2\in \hat{\mathfrak{X}}^2_P(M).\] By the identification in Proposition \ref{Proposition_isomorphic_complexes}, the
hypothesis can be recast as follows
\begin{align*}
[\gamma,\gamma]=0,\ \ [\gamma',\gamma']=0,\ \ \gamma-\gamma'\in \hat{\mathcal{F}}_1,\ \
H^2(\hat{\mathcal{F}}^{\bullet}_{k}/\hat{\mathcal{F}}^{\bullet}_{k+1},d_{\gamma})=0,
\end{align*}
for all $k\geq 1$, where $d_{\gamma}:=ad_{\gamma}$. All these conditions are expressed in terms of the graded Lie algebra
$\mathcal{L}^{\bullet}:=\hat{\mathfrak{X}}^{\bullet+1}_P(M)$, with a complete filtration $\hat{\mathcal{F}}$. Theorem \ref{Teo1} from the
Appendix, shows that there exists a formal vector field $X\in \hat{\mathcal{F}}_1^1$, such that $\gamma=e^{ad_X}(\gamma')$. By Lemma
\ref{gauge_real}, there exists a diffeomorphism $\psi$ of $M$, such that $j^{\infty}_{|P}(\psi^*(W))=e^{ad_X}j^{\infty}_{|P}(W)$, for all
$W\in\mathfrak{X}_P^{\bullet}(M)$. This concludes the proof, since
\[j^{\infty}_{|P}(\psi^*(\pi_2))=e^{ad_X}j^{\infty}_{|P}(\pi_2)=e^{ad_X}(\gamma')=\gamma=j^{\infty}_{|P}(\pi_1).\]

\subsection*{Existence of Poisson structures with a specified infinite jet}

The proof of Theorem \ref{Proposition_formal_equivalence} can be applied to obtain a result on existence of Poisson bivectors with a specified
infinite jet. Let $S$ be a closed embedded submanifold of $M$. An element $\hat{\pi}\in \hat{\mathfrak{X}}^2_S(M)$, satisfying
$[\hat{\pi},\hat{\pi}]=0$, will be called a formal Poisson bivector. Observe that
\[\hat{\pi}_{|S}=[\hat{\pi}] \textrm{ modulo } \hat{\mathcal{F}}_0\in \mathfrak{X}^2(S)\]
gives a Poisson structure on $S$. We will say that $S$ is a symplectic leaf on $\hat{\pi}$, if $\hat{\pi}_{|S}$ is nondegenerate. By the
discussion in section \ref{The first order data}, the first jet of $\hat{\pi}$
\[j^1_{|S}(\hat{\pi})=[\hat{\pi}] \textrm{ modulo } \hat{\mathcal{F}}_1\]
determines an algebroid $A_S$ on $T^*_SM$, and thus can be used to construct a Poisson bivector $\pi^1_S$ on some open neighborhood
$\mathcal{U}$ of $S$, whose first jet coincides with that of $\hat{\pi}$. If the cohomology groups
\[H^{2}(A_S;\mathcal{S}^{k}(TS^{\circ}))\]
vanish for all $k\geq 2$, then by the proof of Theorem \ref{Proposition_formal_equivalence}, there exists a formal vector field $X\in
\hat{\mathcal{F}}_1^1$, such that $e^{ad_X}(j^{\infty}_{|S}\pi^1_S)=\hat{\pi}$. Using now Lemma \ref{gauge_real}, we find a diffeomorphism
$\psi:\mathcal{U}\to\mathcal{U}$, such that
\[j^{\infty}_{|S}(\psi^*(\pi^1_S))=e^{ad_X}(j^{\infty}_{|S}\pi^1_S)=\hat{\pi}.\]
Thus $\pi:=\psi^*(\pi^1_S)$ gives a Poisson structure defined on an open neighborhood of $S$ whose infinite jet is $\hat{\pi}$. Hence we have
proved the following statement.
\begin{corollary}
Let $\hat{\pi}\in \hat{\mathfrak{X}}^2_S(M)$ be a formal Poisson structure, for which $S$ is a symplectic leaf. If the algebroid $A_S$ induced
by $j^1_{|S}\hat{\pi}$ satisfies
\[H^{2}(A_S;\mathcal{S}^{k}(TS^{\circ}))=0, \ \ (\forall)\  k\geq 2,\]
then there exists a Poisson structure $\pi$ defined on some open neighborhood of $S$ such that $\hat{\pi}=j^{\infty}_{|S}\pi$.
\end{corollary}

\section{Proofs of the criteria}

We devote this section to explaining and proving the corollaries from the Introduction. 

\subsection*{Integration of algebroids and differentiable cohomology}

We start by recalling some properties of Lie groupoids and Lie algebroids. For the theory of Lie groupoids, see \cite{Mackenzie, MM}. A Lie
groupoid over a manifold $B$ will be denoted by $\mathcal{G}$, the source and target maps by $s, t: \mathcal{G}\rmap B$ and the unit map by $u:
B\rmap \mathcal{G}$. We will assume all Lie groupoids to be Hausdorff. Every Lie groupoid $\mathcal{G}$ has an associated Lie algebroid
$A(\mathcal{G})$ over $B$, which is the infinitesimal counterpart of $\mathcal{G}$. A Lie algebroid $\mathcal{A}$ is called \textbf{integrable}
if $\mathcal{A}\cong A(\mathcal{G})$ for some Lie groupoid $\mathcal{G}$. The relation between Lie algebroids and Lie groupoids is similar to
that between Lie algebras and Lie groups, the most significant difference is that not every Lie algebroid is integrable.

Recall that a \textbf{transitive Lie algebroid} is a Lie algebroid $\mathcal{A}\to B$ with surjective anchor. For example, if $S\subset M$ is a
symplectic leaf of a Poisson manifold $(M,\pi)$, then the algebroid $A_S$ is a transitive. A groupoid $\mathcal{G}$ is called
\textbf{transitive} if the map $(s,t):\mathcal{G}\to M\times M$ is surjective. The Lie algebroid of a transitive groupoid is transitive.
Conversely, if the base $B$ of a transitive algebroid $\mathcal{A}$ is connected, and $\mathcal{A}$ is integrable, then any Lie groupoid
$\mathcal{G}$ integrating it is transitive. Any transitive groupoid is a \textbf{gauge groupoid} i.e.\ it is of the form $P\times_G P$, where
$G$ is a Lie group and $p:P\to B$ is a principal $G$-bundle. For $P$ one can take any $s$-fiber $s^{-1}(x)$ of $\mathcal{G}$ for $x\in B$, and
$G:=s^{-1}(x)\cap t^{-1}(x)$ (see \cite{Mackenzie} for more details). We can recover $\mathcal{A}$ from $P$ as follows: as a bundle
$\mathcal{A}=TP/G$, the Lie bracket is induced by the identification
\[\Gamma(\mathcal{A})=\mathfrak{X}(P)^G,\]
and the anchor is given by $dp$. Moreover, similar to the theory of Lie groups and Lie algebras, if $\mathcal{A}$ is an integrable transitive
Lie algebroid, then up to isomorphism there is a unique principal bundle $P$ which is connected and simply connected such that $P\times_G P$
integrates $\mathcal{A}$ (see \cite{Mackenzie}). We will also say that $P$ integrates $\mathcal{A}$.

Let $S\subset M$ be a symplectic leaf of a Poisson manifold $(M,\pi)$ and assume that the transitive algebroid $A_S$ is integrable. The
connected and simply connected principal bundle $P\to S$ for which $P\times_G P$ integrates $A_S$ is called \textbf{the Poisson homotopy cover}
of $S$.  We will say that $P$ is smooth, if $A_S$ is integrable (this terminology is justified by the fact that $P$ exists also in the
non-integrable case as a topological principal bundle over $S$ (see \cite{CrFe1})).

Let $\mathcal{A}$ be a transitive algebroid over a connected basis $B$ and denote by $\mathfrak{g}\subset \mathcal{A}$ the kernel of the anchor.
On each fiber of $\mathfrak{g}$ the Lie bracket restricts to a Lie algebra structure $(\mathfrak{g}_x,[\cdot,\cdot]_x)$ and this Lie algebra is
called \textbf{the isotropy Lie algebra at $x$}. In the integrable case, when $\mathcal{A}=A(\mathcal{G})$, the isotropy Lie algebra coincides
with the Lie algebra of the \textbf{isotropy group} $G_x:=s^{-1}(x)\cap t^{-1}(x)$. In the case of a symplectic leaf $S\subset M$ of a Poisson
manifold, for $A_S$ we have that $\mathfrak{g}=TS^{\circ}$.

A Lie groupoid $\mathcal{G}$ is called \textbf{proper}, if $(s,t):\mathcal{G}\to B\times B$ is a proper map.

A \textbf{representation} of a Lie groupoid $\mathcal{G}$ over $B$ is a vector bundle $E\to B$ and a smooth linear action $g:E_x\to E_y$ for
every arrow $g:x\to y$ satisfying the obvious identities. A representation $E$ of $\mathcal{G}$ can be differentiated to a representation of its
Lie algebroid $A(\mathcal{G})$ on the same vector bundle $E$. If the $s$-fibers of $\mathcal{G}$ are connected and simply connected, then every
representation of $A(\mathcal{G})$ comes from a representation of $\mathcal{G}$ (see Proposition 2.2 in \cite{CrFe1}) - and in our applications
this will be usually the case.

The \textbf{differentiable cohomology} of a Lie groupoid $\mathcal{G}$ with coefficients in a representation $E\to B$ is computed by the complex
$\mathcal{C}^p_{\mathrm{diff}}(\mathcal{G};E)$ of smooth maps $c:\mathcal{G}^{(p)}\to E$, where
\[\mathcal{G}^{(p)}:=\{(g_1,\ldots, g_p)\in \mathcal{G}^p | s(g_i)=t(g_{i+1}), i=1,\ldots, p-1\}\]
with $c(g_1,\ldots,g_p)\in E_{t(g_1)}$, and with differential given by
\begin{align*}
dc(g_1, \ldots ,g_{p+1}) &= g_1c(g_2, \ldots , g_{p+1})+\\
& + \sum_{i=1}^p(-1)^i c(g_1, \ldots , g_ig_{i+1}, \ldots , g_{p+1}) + (-1)^{p+1}c(g_1,\ldots,g_p).
\end{align*}
The resulting cohomology groups will be denoted $H^{\bullet}_{\mathrm{diff}}(\mathcal{G},E)$. For more details on this subject see
\cite{Haefliger}.

In the following proposition we list some results from \cite{Cra} which are needed in the proofs of the corollaries from the Introduction.

\begin{proposition}\label{vanishing_proposition}
Let $\mathcal{G}$ be a Lie groupoid over $B$ with Lie algebroid $\mathcal{A}$ and $E\to B$ a representation of $\mathcal{G}$.
\begin{itemize}
\item[(1)] If the $s$-fibers of $\mathcal{G}$ are cohomologically 2-connected, then we have that
\[H^2(\mathcal{A};E)\cong H^2_{\mathrm{diff}}(\mathcal{G};E).\]
\item[(2)] If $\mathcal{G}$ is proper, then $H^2_{\mathrm{diff}}(\mathcal{G};E)=0$.
\item[(3)] If $\mathcal{G}$ is transitive, then
\[H^2_{\mathrm{diff}}(\mathcal{G};E)\cong H^2_{\mathrm{diff}}(\mathcal{G}_x;E_x),\]
where $x\in B$ and $\mathcal{G}_x:=s^{-1}(x)\cap t^{-1}(x)$.
\end{itemize}
\end{proposition}
\begin{proof}
(1) is a particular case of Theorem 4 in \cite{Cra} and (2) follows from Proposition 1 in \cite{Cra}. For (3), since $\mathcal{G}$ is
transitive, it is Morita equivalent to $\mathcal{G}_x$ (see \cite{MM}). By Theorem 1 in \cite{Cra}, a Morita equivalence induces an isomorphism
in cohomology and this proves (3).
\end{proof}

\subsection*{The proofs}
\begin{proof}[Proof of Corollary \ref{C_Dual_ss_Lie}]

Recall that the cotangent algebroid of $(\mathfrak{g}^*,\pi_{{\mathrm{lin}}})$ is isomorphic to the action algebroid $\mathfrak{g}\ltimes
\mathfrak{g}^*\to \mathfrak{g}^*$ for the coadjoint action of $\mathfrak{g}$ on $\mathfrak{g}^*$ and that it is integrable by the action
groupoid $G\ltimes \mathfrak{g}^*$, where $G$ denotes the compact, connected and simply connected Lie group of $\mathfrak{g}$. Moreover, the
symplectic leaves of $(\mathfrak{g}^*,\pi_{\mathrm{lin}})$ are the orbits of the action of $G$. So, because $\mathbb{S}({\mathfrak{g}})$ is $G$
invariant it is a union of symplectic leaves, therefore a Poisson submanifold. The algebroid $A_{\mathbb{S}({\mathfrak{g}})}$ is isomorphic to
the action algebroid $\mathfrak{g}\ltimes \mathbb{S}({\mathfrak{g}})$, therefore it is integrable by the action groupoid $G\ltimes
\mathbb{S}({\mathfrak{g}})$. Since $G$ is simply connected it follows that $H^2_{dR}(G)=0$ (see Theorem 1.14.2 in \cite{DK}). On the other hand
all $s$-fibers of $G\ltimes \mathbb{S}({\mathfrak{g}})$ are diffeomorphic to $G$, and so the assumptions of Proposition
\ref{vanishing_proposition} (1) are satisfied, therefore for any representation $E\to \mathbb{S}({\mathfrak{g}})$ of $G\ltimes
\mathbb{S}({\mathfrak{g}})$ we have that
\[H^2(\mathfrak{g}\ltimes \mathbb{S}({\mathfrak{g}}); E)\cong H^2_{\mathrm{diff}}(G\ltimes \mathbb{S}({\mathfrak{g}});E).\]
Since $G\ltimes \mathbb{S}({\mathfrak{g}})$ is compact it is proper, hence by Proposition \ref{vanishing_proposition} (2), we have that
$H^2_{\mathrm{diff}}(G\ltimes \mathbb{S}({\mathfrak{g}});E)=0$, for every representation $E$. Now the corollary follows from Theorem
\ref{Proposition_formal_equivalence}.
\end{proof}
\begin{proof}[Proof of Corollary \ref{Theorem2}]
Denote by $P$ the Poisson homotopy cover of $S$ with structure group $G$. By hypothesis $P$ is smooth, simply connected and with vanishing
second DeRham cohomology group. Let $\mathcal{G}:=P\times_{G}P$ be the gauge groupoid of $P$. Since every $s$-fiber of $\mathcal{G}$ is
diffeomorphic to $P$, $\mathcal{G}$ satisfies the assumptions of Proposition \ref{vanishing_proposition} (1), therefore
\[H^{2}(A_S;\mathcal{S}^k( TS^{\circ}))\cong H^{2}_{\mathrm{diff}}(\mathcal{G};\mathcal{S}^k(TS^{\circ})).\]
Since $\mathcal{G}$ is transitive, by Proposition \ref{vanishing_proposition} (3), we have that
\[H^{2}_{\mathrm{diff}}(\mathcal{G};\mathcal{S}^k(TS^{\circ}))\cong H^{2}_{\mathrm{diff}}(G;\mathcal{S}^k(T_xS^{\circ})).\]
Since $T_xS^{\circ}\cong \mathfrak{g}$ as $G$ representations (both integrate the adjoint representation of $\mathfrak{g}$), the proof follows
from Theorem \ref{Theorem1}.
\end{proof}
\begin{proof}[Proof of Corollary \ref{Corollary_2}]
This is a direct consequence of Corollary \ref{Theorem2}, because differentiable cohomology of compact groups vanishes (by Proposition
\ref{vanishing_proposition} (2)).
\end{proof}
\begin{proof}[Proof of Corollary \ref{Theorem3}]
Let $x\in S$ and denote by $\mathfrak{g}_x:=T_xS^{\circ}$ the isotropy Lie algebra of the transitive algebroid $A_S$. By hypothesis
$\mathfrak{g}_x$ is reductive, i.e.\ it splits as a direct product of a semi-simple Lie algebra and its center
$\mathfrak{g}_x=\mathfrak{s}_x\oplus \mathfrak{z}_x$, where $\mathfrak{s}_x=[\mathfrak{g}_x,\mathfrak{g}_x]$ and
$\mathfrak{z}_x=Z(\mathfrak{g}_x)$ is the center of $\mathfrak{g}_x$. Since $\mathfrak{g}=TS^{\circ}$ is a Lie algebra bundle, it follows that
this splitting is in fact global:
\[\mathfrak{g}=[\mathfrak{g},\mathfrak{g}]\oplus Z(\mathfrak{g})=\mathfrak{s}\oplus\mathfrak{z}.\]
Since $\mathfrak{s}=[\mathfrak{g},\mathfrak{g}]$ is an ideal of $A_S$, we obtain a short exact sequence of algebroids
\[0\to\mathfrak{s}\to A_S\to A_S^{\textrm{ab}}\to 0,\]
with $A_S^{\textrm{ab}}=A_S/[\mathfrak{g},\mathfrak{g}]$. Similar to the spectral sequence for Lie algebra extensions (see \cite{Serre}), there
is a spectral sequence for extensions of Lie algebroids (see \cite{Mackenzie}, Theorem 5.5 and the remark following it), which in our case
converges to $H^{\bullet}(A_S;\mathcal{S}^{k}(\mathfrak{g}))$, with
\[E_2^{p,q}=H^p(A_S^{\textrm{ab}};H^q(\mathfrak{s};\mathcal{S}^k(\mathfrak{g})))\Rightarrow H^{p+q}(A_S;\mathcal{S}^{k}(\mathfrak{g})).\]
Since $\mathfrak{s}$ is in the kernel of the anchor, $H^q(\mathfrak{s};\mathcal{S}^k(\mathfrak{g}))$ is indeed a vector bundle, with fiber
$H^q(\mathfrak{s};\mathcal{S}^k(\mathfrak{g}))_x=H^q(\mathfrak{s}_x;\mathcal{S}^k(\mathfrak{g}_x))$ and it inherits a representation of
$A_S^{\textrm{ab}}$. Since $\mathfrak{s}_x$ is semi-simple, by the Whitehead Lemma we have that
$H^1(\mathfrak{s}_x;\mathcal{S}^k(\mathfrak{g}_x))=0$ and $H^2(\mathfrak{s}_x;\mathcal{S}^k(\mathfrak{g}_x))=0$. Therefore,
\begin{equation}\label{IncaUna}
H^{2}(A_S;\mathcal{S}^{k}(\mathfrak{g}))\cong H^2(A_S^{\textrm{ab}};\mathcal{S}^k(\mathfrak{g})^{\mathfrak{s}}),
\end{equation}
where $\mathcal{S}^k(\mathfrak{g}_x)^{\mathfrak{s}_x}$ is the $\mathfrak{s}_x$ invariant part of $\mathcal{S}^k(\mathfrak{g}_x)$. The hypothesis
of the theorem tell us that $A_S^{\textrm{ab}}$ is integrable by a principle bundle $P^{\mathrm{ab}}$ which is simply connected, has vanishing
second DeRham cohomology and compact structure group $T$. Therefore, by (\ref{IncaUna}) and by applying Proposition \ref{vanishing_proposition}
(1), (2) and (3) we obtain that
\begin{align*}
H^{2}(A_S;\mathcal{S}^{k}(\mathfrak{g}))&\cong H^2(A_S^{\textrm{ab}};\mathcal{S}^k(\mathfrak{g})^{\mathfrak{s}})\cong
H^2_{\mathrm{diff}}(P^{\mathrm{ab}}\times_{T}P^{\mathrm{ab}};\mathcal{S}^k(\mathfrak{g})^{\mathfrak{s}}) \cong \\
&\cong H^2_{\mathrm{diff}}(T;\mathcal{S}^k(\mathfrak{g}_x)^{\mathfrak{s}_x})=0.
\end{align*}
Thus Theorem \ref{Theorem1} finishes the proof.
\end{proof}

\begin{proof}[Proof of Corollary \ref{Corollary_3}]
Assume that $\mathfrak{g}_{x}$ is semi-simple, $\pi_1(S,x)$ is finite and $\pi_2(S,x)$ is a
 torsion group. With the notation from above, we have that $A_S^{\textrm{ab}}=TS$. $TS$ is integrable
and the simply connected principal bundle integrating it is $\widetilde{S}$, the universal cover of $S$. Finiteness of $\pi_1(S)$ is equivalent
to compactness of the structure group of $\widetilde{S}$. By the Hurewicz theorem $H_2(\widetilde{S},\mathbb{Z})\cong \pi_2(\widetilde{S})$ and
since $\pi_2(\widetilde{S})=\pi_2(S)$ is torsion, we have that $H^2_{\mathrm{dR}}(\widetilde{S})=0$. So the result follows from Corollary
\ref{Theorem3}.
\end{proof}

\section*{Appendix: Equivalence of MC-elements in complete GLA's}

In this Appendix we develop some general facts about graded Lie algebras with a complete filtration, with the aim of proving Theorem \ref{Teo1}
which was used in the proof of Proposition \ref{Proposition_formal_equivalence}. Some of the constructions given here can be also found in
Appendix B.1 of \cite{Bursztyn} in the more general setting of differential graded Lie algebra with a complete filtration. In fact all our
constructions can be adapted to this setup, in particular also Theorem \ref{Teo1}. The analog of Theorem \ref{Teo1}, in the case of differential graded associative
algebras can be found in the Appendix A of \cite{CMB}.

\begin{definitions}
\textbf{(1)} A \textbf{graded Lie algebra} $(\mathcal{L}^{\bullet},[\cdot,\cdot])$ (\textbf{GLA}) consists of a $\mathbb{Z}$-graded vector space
$\mathcal{L}^{\bullet}$ endowed with a graded bracket $[\cdot,\cdot]:\mathcal{L}^{p}\times \mathcal{L}^{q}\to\mathcal{L}^{p+q}$, satisfying
\begin{itemize}
\item[-]graded antisymmetry: $[X,Y]=-(-1)^{|X||Y|}[Y,X]$,
\item[-]the graded Jacobi identity:
$[X,[Y,Z]]=[[X,Y],Z]+(-1)^{|X||Y|}[Y,[X,Z]]$.
\end{itemize}

\textbf{(2)} An element $\gamma\in \mathcal{L}^1$ is called a \textbf{Maurer Cartan element} (\textbf{MC-element}) if it satisfies
$[\gamma,\gamma]=0$.

\textbf{(3)} A \textbf{filtration} on a GLA is a decreasing sequence of homogeneous ideals $\mathcal{F}_n\mathcal{L}^{\bullet}$
\[\mathcal{L}^{\bullet}\supset\mathcal{F}_{0}\mathcal{L}^{\bullet}\supset \ldots \supset\mathcal{F}_n\mathcal{L}^{\bullet}\supset \mathcal{F}_{n+1}\mathcal{L}^{\bullet}\supset\ldots,\]
satisfying
\[[\mathcal{F}_{n}\mathcal{L},\mathcal{F}_{m}\mathcal{L}]\subset
\mathcal{F}_{n+m}\mathcal{L},\ \ [\mathcal{L},\mathcal{F}_{n}\mathcal{L}]\subset \mathcal{F}_{n}\mathcal{L}.\]

\textbf{(4)} A filtration $\mathcal{F}\mathcal{L}$ is called \textbf{complete}, if $\mathcal{L}$ is isomorphic to the projective limit
$\varprojlim \mathcal{L}/\mathcal{F}_{n}\mathcal{L}$.
\end{definitions}

An example of a GLA with a complete filtration was constructed in Section \ref{Formal equivalence of Poisson structures around symplectic
leaves}: starting from a manifold $M$ and a closed embedded submanifold $P\subset M$, we constructed
$(\hat{\mathfrak{X}}^{\bullet+1}_P(M),[\cdot,\cdot])$, the algebra of formal vector fields along $P$, with filtration given by the powers of the
vanishing ideal of $P$.

Let $\mathcal{L}$ be a GLA with a complete filtration $\mathcal{FL}$. Completeness implies that
\[\bigcap_{n\geq 0}\mathcal{F}_n\mathcal{L}=0.\]
To handle the filtration better, we introduce the following function on $\mathcal{L}$, which we will call the $L$-norm of the filtration
\begin{equation}\label{EQ2}
\|\cdot\|:\mathcal{L}\to \mathbb{R}_{\geq 0},\quad \|X\|=\left\{\begin{array}{ccc}1, &\textrm{if}& X\in \mathcal{L}\backslash
\mathcal{F}_{0}\mathcal{L},\\ \frac{1}{2^{n}}, &\textrm{if}& X\in \mathcal{F}_n\mathcal{L}\backslash \mathcal{F}_{n+1}\mathcal{L},\\
0, &\textrm{if} & X=0.\end{array}\right.
\end{equation}
Properties of the filtration are translated into the following properties of the $L$-norm, which we will often use:
\begin{itemize}
\item $\|X\|=0$ if and only if $X=0$,
\item $\|X+Y\|\leq \|X\|\vee\|Y\|$\footnote{$u\vee v$ denotes $\mathrm{max}\{u,v\}$.},
\item $\|\alpha X\|\leq \|X\|, \ \ (\forall) \alpha\in\mathbb{R}$,
\item $\|[X,Y]\|\leq \|X\|\|Y\|$.
\end{itemize}
The first two properties show that $d(X,Y)=\|X-Y\|$ defines a metric on $\mathcal{L}$. The completeness of the filtration is easily seen to be
equivalent to the completeness of the metric space $(\mathcal{L},d)$. Convergent series have a very simple description, as the following
straightforward Lemma shows.
\begin{lemma}\label{LemConv}
A series $\sum_n x_n$, $x_n\in\mathcal{L}$ converges if and only if $x_n\stackrel{n\to\infty}{\to}0$.
\end{lemma}


Note that $\mathfrak{g}(\mathcal{L}):=\mathcal{F}_{1}\mathcal{L}^{0}$ forms a Lie subalgebra of $\mathcal{L}^0$. 
Elements $X\in\mathfrak{g}(\mathcal{L})$ satisfy $\|ad_X(Y)\|\leq \frac{1}{2}\|Y\|$, for all $Y\in\mathcal{L}$, therefore (by Lemma
\ref{LemConv}) the exponential of $ad_X$ converges, and it defines an GLA-automorphism of $\mathcal{L}^{\bullet}$, denoted by
\[Ad(e^X):\mathcal{L}^{\bullet}\to\mathcal{L}^{\bullet},\quad Ad(e^X)Y:=e^{ad_X}(Y)=\sum_{n\geq 0}\frac{ad_X^n}{n!}(Y).\]
By Lemma \ref{LemConv}, the Campbell-Hausdorff formula converges for all $X,Y\in\mathfrak{g}(\mathcal{L})$
\begin{eqnarray}\label{Campbell-Hausdorff}
&&X*Y=X+Y+\sum_{k\geq 1}\frac{(-1)^{k}}{k+1}D_k(X,Y),\ \ \textrm{where}\\
\nonumber && D_k(X,Y)=\sum_{ l_i + m_i > 0}\frac{ad_X^{\phantom{o}l_1}}{l_1!}\circ\frac{ad_Y^{\phantom{o}m_1}}{m_1!}\circ\ldots
\circ\frac{ad_X^{\phantom{o}l_k}}{l_k!}\circ\frac{ad_Y^{\phantom{o}m_k}}{m_k!}(X).
\end{eqnarray}
We will use the notation $\mathcal{G}(\mathcal{L})=\{e^X| X\in \mathfrak{g}(\mathcal{L})\}$, i.e.\ $\mathcal{G}(\mathcal{L})$ is the same
(topological) space as $\mathfrak{g}(\mathcal{L})$, but we just denote its elements by $e^X$. The universal properties of the Campbell-Hausdorff
formula (\ref{Campbell-Hausdorff}), show that $\mathcal{G}(\mathcal{L})$ endowed with the product $e^Xe^Y=e^{X*Y}$ forms a topological group,
and moreover $Ad$ gives an action of $\mathcal{G}(\mathcal{L})$ on $\mathcal{L}$ by continuous graded Lie algebra automorphisms. More precisely
we have the following properties.
\begin{proposition}\label{PropGroup}
The group $\mathcal{G}(\mathcal{L})$ and the map $Ad$ satisfy the following
\begin{itemize}
\item[(a)] multiplication is continuous: $\|X*Y-X'*Y'\|\leq \|X-X'\|\vee\|Y-Y'\|$,
\item[(b)] $Ad$ is a representation: $Ad(e^Xe^Y)=Ad(e^X)\circ Ad(e^Y)$,
\item[(c)] $Ad$ preserves the grading: $Ad(e^X)\mathcal{L}^p=\mathcal{L}^p$,
\item[(d)] $Ad$ commutes with the bracket: $Ad(e^X)[U,V]=[Ad(e^X)U,Ad(e^X)V]$,
\item[(e)] $Ad$ preserves the $L$-norm: $\|Ad(e^X)U\|=\|U\|$,
\item[(f)] $Ad$ is continuous: $\|Ad(e^X)U-Ad(e^Y)V\|\leq
\|U-V\|\vee\|X-Y\|\|V\|$,
\end{itemize}
where $X,Y,X',Y'\in \mathfrak{g}(\mathcal{L})$ and $U,V\in\mathcal{L}^{\bullet}$.
\end{proposition}

Let $\gamma$ be an MC-element. Notice that $[\gamma,\gamma]=0$, implies that $d_{\gamma}:=ad_{\gamma}$ is a differential on
$\mathcal{L}^{\bullet}$. The fact that $\mathcal{F}_k\mathcal{L}$ are ideals implies that $(\mathcal{F}_k\mathcal{L}^{\bullet},d_{\gamma})$ are
subcomplexes of $(\mathcal{L}^{\bullet},d_{\gamma})$. The induced differential on the consecutive complexes depends only on $\gamma$ modulo
$\mathcal{F}_1$, and their cohomology groups will be denoted
\[H^n_{\gamma}(\mathcal{F}_{k}\mathcal{L}^{\bullet}/\mathcal{F}_{k+1}\mathcal{L}^{\bullet}).\]

For $e^X\in\mathcal{G}(\mathcal{L})$, $Ad(e^X)\gamma$ is again an MC-element, and $\gamma$ and $Ad(e^X)\gamma$ are called \textbf{gauge
equivalent}. We will use a linear approximation of the action $\mathcal{G}(\mathcal{L})$ on MC-elements.
\begin{lemma}\label{LemAprox}
For $\gamma$ an MC-element and $e^X\in\mathcal{G}(\mathcal{L})$, we have that
\[\|Ad(e^X)\gamma-\gamma+d_{\gamma}X\|\leq\|X\|^2.\]
\end{lemma}

We have the following criterion for gauge equivalence.
\begin{theorem}\label{Teo1}
Let $(\mathcal{L}^{\bullet},[\cdot,\cdot])$ be a GLA with a complete filtration $\mathcal{F}_n\mathcal{L}$ and $\gamma,\gamma'$ be two Maurer
Cartan elements. If $\gamma-\gamma'\in\mathcal{F}_p\mathcal{L}^{1}$, for some $p\geq 1$ and for all $q\geq p$ we have
\[H^1_{\gamma}(\mathcal{F}_{q}\mathcal{L}^{\bullet}/\mathcal{F}_{q+1}\mathcal{L}^{\bullet})=0,\]
then $\gamma$ and $\gamma'$ are gauge equivalent, i.e.\ there exists an element $e^X\in \mathcal{G}(\mathcal{L})$ such that
$\gamma=Ad(e^X)\gamma'$.
\end{theorem}

\begin{proof}
By hypothesis, for $q\geq p$, we can find homotopy operators $h^q_1:\mathcal{F}_{q}\mathcal{L}^1\to \mathcal{F}_{q}\mathcal{L}^0$ and
$h^q_2:\mathcal{F}_{q}\mathcal{L}^2\to \mathcal{F}_{q}\mathcal{L}^1$ such that
\[h^q_1(\mathcal{F}_{q+1}\mathcal{L}^1)\subset \mathcal{F}_{q+1}\mathcal{L}^0,\ \ h^q_2(\mathcal{F}_{q+1}\mathcal{L}^2)\subset \mathcal{F}_{q+1}\mathcal{L}^1,\]
\[\textrm{ and }
(d_{\gamma}h^q_1+h^q_2d_{\gamma}-Id)(\mathcal{F}_{q}\mathcal{L}^1)\subset \mathcal{F}_{q+1}\mathcal{L}^1.\] We will first prove an estimate. Let
$q\geq p$ and $\tilde{\gamma}$ a MC-element, with $\|\tilde{\gamma}-\gamma\|\leq \frac{1}{2^{q}}$. Then for
${\widetilde{X}}:=h^{q}_1(\tilde{\gamma}-\gamma)$, we claim that the following estimates hold:
\begin{equation}\label{EQ3}
\|{\widetilde{X}}\|\leq \frac{1}{2^q}, \ \ \|Ad(e^{\widetilde{X}})\tilde{\gamma}-\gamma\|\leq \frac{1}{2^{q+1}}.
\end{equation}
The first estimate follows directly from the fact that ${\widetilde{X}}=h^{q}_1(\tilde{\gamma}-\gamma)\in \mathcal{F}_q\mathcal{L}^0$. To prove
the second we compute:
\begin{align*}
\|Ad(e^{\widetilde{X}})\tilde{\gamma}-\gamma\|&\leq
\|Ad(e^{\widetilde{X}})\tilde{\gamma}-\tilde{\gamma}+d_{\tilde{\gamma}}({\widetilde{X}})\|\vee\|\tilde{\gamma}-d_{\tilde{\gamma}}({\widetilde{X}})-\gamma\|\leq\\
&\leq\|{\widetilde{X}}\|^{2}\vee\|\tilde{\gamma}-\gamma-d_{\gamma}({\widetilde{X}})\|\vee\|[\gamma-\tilde{\gamma},{\widetilde{X}}]\|\leq\\
&\leq\|{\widetilde{X}}\|^{2}\vee\|\tilde{\gamma}-\gamma-d_{\gamma}({\widetilde{X}})\|\vee\|\gamma-\tilde{\gamma}\|\|{\widetilde{X}}\|,
\end{align*}
where for the second inequality we have used Lemma \ref{LemAprox}. Since $\|\tilde{\gamma}-\gamma\|\leq \frac{1}{2^{q}}$, we obtain
\begin{equation}\label{EQ4}
\|Ad(e^{\widetilde{X}})\tilde{\gamma}-\gamma\| \leq \frac{1}{2^{2q}}\vee\|(Id-d_{\gamma}h^{q}_1)(\tilde{\gamma}-\gamma)\|.
\end{equation}
On the other hand, we have that
\begin{align*}
\|(Id-d_{\gamma}h^{q}_1)(\tilde{\gamma}-\gamma)\|&\leq \|(Id-d_{\gamma}h^{q}_1-h^{q}_2d_{\gamma})(\tilde{\gamma}-\gamma)\|\vee
\|h^{q}_2(d_{\gamma}(\tilde{\gamma}-\gamma))\|\leq\\
&\leq \frac{1}{2^{q+1}}\vee \|h^{q}_2(d_{\gamma}(\tilde{\gamma}-\gamma))\|.
\end{align*}
Observe that $d_{\gamma}(\tilde{\gamma}-\gamma)=-\frac{1}{2}[\tilde{\gamma}-\gamma,\tilde{\gamma}-\gamma]$, therefore
$\|d_{\gamma}(\tilde{\gamma}-\gamma)\|\leq \frac{1}{2^{2q}}\leq \frac{1}{2^{q+1}}$, and so we obtain
\begin{equation}\label{EQ5}
\|(Id-d_{\gamma}h^{q}_1)(\tilde{\gamma}-\gamma)\|\leq \frac{1}{2^{q+1}}.
\end{equation}
So, (\ref{EQ4}) and (\ref{EQ5}) imply the second part of (\ref{EQ3}).

We construct a sequence of MC-elements $\{\gamma_k\}_{k\geq 0}$ and a sequence of group elements $\{e^{X_k}\}_{k\geq
1}\in\mathcal{G}(\mathcal{L})$ by the following recursive formulas:
\begin{align*}
\gamma_0&:=\gamma',\\
X_{k}&:=h_1^{p+k-1}(\gamma_{k-1}-\gamma), \ \ \textrm{ for }k\geq 1,\\
\gamma_{k}&:=Ad(e^{X_{k}})\gamma_{k-1},  \ \ \textrm{ for }k\geq 1.
\end{align*}
To show that this formulas give indeed well-defined sequences, we have to check that $\gamma_{k-1}-\gamma\in\mathcal{F}_{p+k-1}\mathcal{L}^1$.
This holds for $k=1$, and in general it follows by applying inductively at each step $k\geq 1$ the estimate (\ref{EQ3}) to
$\widetilde{\gamma}=\gamma_{k-1}$, and $q=p+k-1$, to obtain the inequalities:
\[\|X_{k}\|\leq \frac{1}{2^{p+k-1}}, \ \ \|\gamma_{k}-\gamma\|\leq \frac{1}{2^{p+k}}.\]

Using Proposition \ref{PropGroup} (a) we obtain
\[\|X_k*X_{k-1}\ldots*X_1-X_{k-1}\ldots*X_1\|\leq \|X_k\|\leq \frac{1}{2^{p+k-1}},\]
therefore by Lemma \ref{LemConv}, the product $X_k*X_{k-1}*\ldots*X_1$ converges to some element $X$. Applying Proposition \ref{PropGroup} (a)
$k$ times, we obtain
\[\|X_k*X_{k-1}\ldots*X_1\|\leq \|X_{k}\|\vee\|X_{k-1}\|\vee\ldots \vee\|X_1\|\leq \frac{1}{2^p},\]
therefore  $\|X\|<1$, thus $X\in\mathfrak{g}(\mathcal{L})$. On the other hand, we have that
\begin{eqnarray*}
\|Ad(e^X)\gamma'-\gamma\|&\leq& \|Ad(e^X)\gamma'-\gamma_{k}\|\vee \|\gamma_{k}-\gamma\|\leq\\
&\leq& \|Ad(e^X)\gamma'-Ad(e^{X_k*\ldots
*X_1})\gamma'\|\vee
\frac{1}{2^{p+k}}\leq\\
&\leq&\|X-X_k*\ldots *X_1\|\vee \frac{1}{2^{p+k}},
\end{eqnarray*}
where for the last estimate we have used Proposition \ref{PropGroup}
(f). If we let $k\to \infty$ we obtain the conclusion:
$Ad(e^X)\gamma'=\gamma$.
\end{proof}

\bibliographystyle{amsplain}
\def\lllll{}


\end{document}